\documentclass{article}%
\usepackage{amssymb}
\usepackage{amsmath}
\usepackage{amsfonts}
\usepackage{graphicx}%
\providecommand{\U}[1]{\protect\rule{.1in}{.1in}}
\newtheorem{theorem}{Theorem}

\newtheorem{proposition}[theorem]{Proposition}

\newenvironment{proof}[1][Proof]{\noindent\textbf{#1.} }{\ \rule{0.5em}{0.5em}}
\begin{document}

\title{\textbf{ON GENERALIZED EULER SPIRALS IN }$E^{3}$}
\author{\textbf{Semra SARACOGLU}\thanks{Siirt University, Faculty of Science and Arts,
Department of Mathematics, Siirt, TURKEY}}
\maketitle

\begin{abstract}
The Cornu spirals on plane are the curves whose curvatures are linear.
Generalized planar cornu spirals and Euler spirals in $E^{3}$, the curves
whose curvatures are linear are defined in [1,5]. In this study, these curves
are presented as the ratio of two rational linear functions.

Also here, generalized Euler spirals in $E^{3}$ has been defined and given
their some various characterizations. The approach I used in this paper is
useful in understanding the role of Euler spirals in $E^{3}$ in differential geometry.

\textbf{AMS Subj. Class.:} 53A04, 53A05.


\end{abstract}

\section{INTRODUCTION}

Spirals are the curves that had been introduced in the 1700s. Privately, one
of the most fascinating spiral in nature and science is Euler Spiral. This
curve is defined by the main property that its curvature is equal to its arclength.

Euler Spirals were discovered independently by three researchers [5, 9]. In
1694, Bernoulli wrote the equations for the Euler spiral for the first time,
but did not draw the spirals or compute them numerically. In 1744, Euler
rediscovered the curve's equations, described their properties, and derived a
series expansion to the curve's integrals. Later, in 1781, he also computed
the spiral's end points. The curves were re-discovered in 1890 for the third
time by Talbot, who used them to design railway tracks [5].

On the other hand, the Euler spiral, defined by the linear relationship
between curvature and arclength, was first proposed as a problem of elasticity
of James Bernoulli, then solved accurately by Leonhard Euler [9]. The Euler
spiral, also well known as Clothoid or Cornu Spiral is a plane curve and
defined as the curve in which the curvature increases linearly with arclength.
Changing the constant of proportionality merely scales the entire curve
[9,11]. This curve is known as an example of such an aesthetic curve and also
we know that its curvature varies linearly with arclength [2,5,7, 9]. In [5],
their proposed curve has both its curvature and torsion change linearly with length.

In this paper, at the beginning, I give the basic concepts about the study
then I deal with Euler spirals whose curvature and torsion are linear and
generalized Euler spirals whose ratio between its curvature and torsion
changes linearly in $E^{3}.$

Here, I give the definitions of the spirals in $E^{2}$ also referred to as a
Cornu Spiral, then the spirals in $E^{3}$ with the name Euler Spirals. In
addition, by giving the definition of logarithmic spirals in $E^{3},$ we seek
that if the logarithmic spirals in $E^{3}$ are generalized Euler spirals or
not. In [5], the curvature and torsion are taken linearly; by using some
characterizations of [5] and also differently from [5], we have presented the
ratio of the curvature and torsion change linearly. And finally, I have named
these spirals as generalized euler spirals in $E^{3}$ by giving some different characterizations.

\section{PRELIMINARIES}

Now, the basic concepts have been recalled on classical differential geometry
of space curves. References [1, 5, 6, 10] contain the concepts about the
elements and properties of Euler spirals in $E^{2}$ and in $E^{3}.$ Then we
give the definition of generalized Euler Spirals in $E^{3}.$

In differential geometry of a regular curve in Euchlidean 3-space, it is
well-known that, one of the important problems is characterization of a
regular curve. The curvature $\kappa$ and the torsion $\tau$ of a regular
curve play an important role to determine the shape and size of the curve [5,10]

Let%

\begin{equation}
\alpha:I\rightarrow E^{3}%
\end{equation}%
\[
\text{\ \ \ \ \ \ \ }s\mapsto\alpha(s)
\]
be unit speed curve and $\left\{  T,N,B\right\}  $ Frenet frame of $\alpha.$
$T,N,B$ are the unit tangent, principal normal and binormal vectors
respectively. Let $\kappa$ and $\tau$ be the curvatures of the curve $\alpha.$

A spatial curve $\alpha(s)$ is determined by its curvature $\kappa(s)$ and its
torsion $\tau(s).$ Intuitively, a curve can be obtained from a straight \ line
by bending (curvature) and twisting (torsion). The Frenet Serret Equations
will be necessary in the derivation of Euler spirals in $E^{3}.$ In the
following%
\[
\overrightarrow{T}(s)=\frac{d\alpha}{ds}%
\]
is the unit tangent vector, $\overrightarrow{N}(s)$ is the unit normal vector,
and $\overrightarrow{B}(s)=\overrightarrow{T}(s)\times\overrightarrow{N}(s)$
is the binormal vector. We assume an arc-length parametrization.

Given a curvature $\kappa(s)$ $\rangle$ $0$ and a torsion $\tau(s),$ according
to the fundamental theorem of the local theory of curves [10], there exists a
unique (up to rigid motion) spatial curve, parametrized by the arc-length $s,$
defined by its Frenet-Serret equations, as follows:%
\begin{align}
\frac{d\overrightarrow{T}(s)}{ds}  &  =\kappa(s)\overrightarrow{N}(s),\\
\frac{d\overrightarrow{N}(s)}{ds}  &  =-\kappa(s)\overrightarrow{T}%
(s)+\tau(s)\overrightarrow{B}(s)\nonumber\\
\frac{d\overrightarrow{B}(s)}{ds}  &  =-\tau(s)\overrightarrow{N}(s).\nonumber
\end{align}

From [1], a unit speed curve which lies in a surface is said to be a Cornu
spiral if its curvatures are non-constant linear function of the natural
parameter $s,$ that is if $\alpha(s)=as+b,$ with $a\neq0.$ The classical Cornu
spiral in $E^{2}$ was studied by J. Bernoulli and it appears in diffraction.

Harary and Tall define the Euler spirals in $E^{3}$ the curve having both its
curvature and torsion evolve linearly along the curve $(Figure.1).$
Furthermore, they require that their curve conforms with the definition of a
$Cornu$ $spiral.$%

\[%
{\includegraphics[
natheight=4.587000in,
natwidth=12.788000in,
height=0.947in,
width=2.1534in
]%
{LSFQ3R00.wmf}%
}%
\]

\[
Figure.1\text{ Euler Spiral in }E^{3}\text{ }[5].
\]
Here, these curves whose curvatures and torsion evolve linearly are called
Euler spirals in $E^{3}$. Thus for some constants $a,b,c,d$ $\epsilon$ $%
\mathbb{R}
,$%
\begin{align}
\kappa(s)  &  =as+b\\
\tau(s)  &  =cs+d\nonumber
\end{align}

On the other hand, in [13], another property of Euler spirals is underlined as
: Euler spirals have a linear relation between the curvature and torsion, it
is a special case of Bertrand curves.

An attempt to generalize Euler spirals to $3D,$ maintaining the linearity of
the curvature, is presented in [5, 8]. A given polygon is refined, such that
the polygon satisfies both arc-length parametrization and linear distribution
of the $discrete$ $curvature$ $binormal$ $vector.$ The logarithm ignores the
torsion, despite being an important characteristic of $3D$ curves [5]. Harary
and Tal prove that their curve satisfies properties that characterize fair and
appealing curves and reduces to the $2D$ Euler spiral in the planar case.

Accordingly the Euler spirals in $E^{2}$ and in $E^{3}$ that satisfy
$(Equation.3)$ by a set of differential equations is the curve $\alpha$ for
which the following conditions hold:%
\begin{align*}
\frac{d\overrightarrow{T}(s)}{ds}  &  =\left(  \frac{1}{(as+b)}\right)
\overrightarrow{N}(s),\\
\frac{d\overrightarrow{N}(s)}{ds}  &  =-\left(  \frac{1}{(as+b)}\right)
\overrightarrow{T}(s)+\left(  \frac{1}{(cs+d)}\right)  \overrightarrow{B}(s)\\
\frac{d\overrightarrow{B}(s)}{ds}  &  =-\left(  \frac{1}{(cs+d)}\right)
\overrightarrow{N}(s).
\end{align*}
Next, we define logarithmic spiral having a linear radius of curvature and a
linear radius of torsion from [6]. They seek a spiral that has both a linear
radius of curvature and a linear radius of torsion in the arc-length
parametrization $s:$%
\begin{align*}
\kappa(s)  &  =\frac{1}{as+b}\\
\tau(s)  &  =\frac{1}{cs+d}%
\end{align*}
where $a,b,c$ and $d$ are constants.

\[%
{\includegraphics[
natheight=5.412900in,
natwidth=11.604000in,
height=1.094in,
width=2.3367in
]%
{LSBV8N00.wmf}%
}%
\]%
\[
Figure.2\text{ Logarithmic spiral in }E^{3}\text{ }[6].
\]
In addition to these, we want to give our definition that Euler spirals in
$E^{3}$ whose ratio between its curvature and torsion evolve linearly is
called generalized Euler spirals in $E^{3}.$ Thus for some constants $a,b,c,d$
$\epsilon$ $%
\mathbb{R}
,$%
\[
\dfrac{\kappa}{\tau}=\frac{as+b}{cs+d}%
\]

\section{EULER SPIRALS IN $E^{3}$}

In this section, we study some characterizations of Euler spirals in $E^{3}$
by giving some theorems by using the definitions in $section.2.$

\begin{proposition}
If the curvature $\tau$ is zero then $\kappa=as+b$ and then the curv\-e is
planar cornu spiral.

\begin{proof}
If $\tau=0$ and the curvature is linear, then the ratio%
\[
\frac{\tau}{\kappa}=0
\]

\end{proof}
\end{proposition}

Therefore, we see that the curve is planar cornu spiral.

\begin{proposition}
If the curvatures are
\begin{align*}
\tau &  =as+b\\
\kappa &  =c
\end{align*}
then the euler spirals are rectifying curves.

\begin{proof}
If we take the ratio%
\[
\frac{\tau}{\kappa}=\frac{as+b}{c}%
\]
where $\lambda_{1}$ and $\lambda_{2}$ , with $\lambda_{1}\neq0$ are constants,
then
\[
\frac{\tau}{\kappa}=\lambda_{1}s+\lambda_{2}.
\]
It shows us that the Euler spirals are rectifying curves. It can be easily
seen from [3] that rectifying curves have very simple characterization in
terms of the ratio $\dfrac{\tau}{\kappa}.$
\end{proof}
\end{proposition}

\begin{proposition}
Euler spirals in $E^{3}$are Bertrand curves.

\begin{proof}
From \ the definition of Euler spiral in $E^{3}$ and the equation $\left(
3\right)  $ in $Section.2,$ the equations can be taken as%
\begin{align*}
\tau(s)  &  =c_{1}s+c_{2}\\
\kappa(s)  &  =d_{1}s+d_{2},\text{ with }c_{1}\neq0\text{ and }d_{1}\neq0.
\end{align*}
Here,%
\[
s=\frac{1}{c_{1}}(\tau-c_{2})
\]
and then,%
\begin{align*}
\kappa &  =\frac{d_{1}}{c_{1}}(\tau-c_{2})+d_{2}\\
c_{1}\kappa &  =d_{1}(\tau-c_{2})+c_{1}d_{2}\\
c_{1}\kappa+d_{1}\tau &  =c_{3}\\
\frac{c_{1}}{c_{3}}\kappa+\frac{d_{1}}{c_{3}}\tau &  =1
\end{align*}
Thus, we obtain%
\[
\lambda\kappa+\mu\tau=1
\]
It shows us that there is a Bertrand curve pair that corresponds to Euler
spirals in $E^{3}$.
\end{proof}
\end{proposition}

\begin{theorem}
Let $M$ be a surface in $E^{3}$ and $\alpha:I\rightarrow M$ be unit speed
curve but not a general helix. If the Darboux curve
\[
W(s)=\tau T+\kappa B
\]
is geodesic curve on the surface $M,$ then the curve $\alpha$ is Euler spirals
in $E^{3}.$
\end{theorem}

We have
\begin{align}
W(s)  &  =\tau T+\kappa B\\
W^{\prime}(s)  &  =\tau^{\prime}T+\kappa^{\prime}B\\
W^{\prime\prime}(s)  &  =\tau^{\prime\prime}T+\kappa^{\prime\prime}%
B+(\kappa\tau^{\prime}-\tau\kappa^{\prime})N\\
W^{\prime\prime}(s)  &  =\tau^{\prime\prime}T+\kappa^{\prime\prime}B+\left[
\left(  \frac{\tau}{\kappa}\right)  ^{\prime}\kappa^{2}\right]  N
\end{align}
Here, $n(s)$ is the unit normal vector field of the surface $M.$ Darboux curve
is geodesic on surface $M,$ therefore
\[
W^{\prime\prime}(s)=\lambda(s)n(s).
\]
And also $n=N,$ \ then we have
\[
W^{\prime\prime}(s)=\lambda(s)N(s).
\]
From (7), if we take
\[
\tau^{\prime\prime}=0,\text{ }\kappa^{\prime\prime}=0
\]
and also%
\begin{align*}
\kappa &  =as+b\\
\tau &  =cs+d
\end{align*}
then the curve $\alpha$ is generalized Euler spiral in $E^{3}.$

\section{GENERALIZED EULER SPIRALS IN $E^{3}$}

In this section, we investigate generalized Euler spirals in $E^{3}$ by using
the definitions in $section.2.$

\begin{theorem}
In $E^{3},$ all logarithmic spirals are generalized euler spirals.

\begin{proof}
As it is known that in all logarithmic spirals, the curvatures are linear as:%
\begin{align*}
\kappa(s)  &  =\frac{1}{as+b}\\
\tau(s)  &  =\frac{1}{cs+d}%
\end{align*}
In that case, it is clear that the ratio between the curvatures can be given
as:%
\[
\dfrac{\kappa}{\tau}=\frac{cs+d}{as+b}%
\]
Thus, it can be easily seen all logarithmic spirals are generalized euler spirals.

\begin{theorem}
Euler spirals are generalized Euler spirals in $E^{3}.$

\begin{proof}
It is clear from the property of curvature, torsion and the ratio that are
linear as:%
\begin{align*}
\kappa(s)  &  =as+b\\
\tau(s)  &  =cs+d
\end{align*}
and then%
\[
\dfrac{\kappa}{\tau}=\frac{as+b}{cs+d}%
\]
That shows us Euler spirals are generalized Euler spirals.
\end{proof}
\end{theorem}
\end{proof}
\end{theorem}

\begin{proposition}
All generalized Euler spirals in $E^{3}$ that have the property
\[
\dfrac{\tau}{\kappa}=d_{1}s+d_{2}%
\]
are rectifying curves.

\begin{proof}
If the curvatures $\kappa(s)$ and $\tau(s)$ are taken as
\begin{align*}
\kappa(s)  &  =c\\
\tau(s)  &  =d_{1}s+d_{2}\text{ with }d_{1}\neq0
\end{align*}
then%
\[
\dfrac{\tau}{\kappa}=\frac{d_{1}s+d_{2}}{c}=\lambda_{1}s+\lambda_{2}\text{ }%
\]
where $\lambda_{1}$ and $\lambda_{2}$ are constants. This gives us that if the
curve $\alpha$ is generalized Euler spiral then it is also in rectifying plane.
\end{proof}
\end{proposition}

\textbf{Result.1 }General helices are generalized euler spirals.

\begin{proof}
It \ can be seen from the property of curvatures that are linear and the ratio
is also constant as it is shown:%
\[
\dfrac{\tau}{\kappa}=\lambda
\]

\end{proof}

\begin{theorem}
Let
\begin{equation}
\alpha:I\rightarrow E^{3}%
\end{equation}%
\[
\text{\ \ \ \ \ \ \ }s\mapsto\alpha(s)
\]
be unit speed curve and let $\kappa$ and $\tau$ be the curvatures of the
Frenet vectors of the curve $\alpha.$ For $a,b,c,d,\lambda\in%
\mathbb{R}
,$ let take the curve $\beta$ as%
\begin{equation}
\beta(s)=\alpha(s)+(as+b)T+(cs+d)B+\lambda N.
\end{equation}
In this case, the curve $\alpha$ is generalized Euler spiral which has the
property%
\[
\frac{\kappa}{\tau}=\frac{cs+d}{as+b}%
\]
if and only if the curves $\beta$ and $(T)$ are the involute-evolute pair.
Here, the curve $(T)$ is the tangent indicatrix of $\alpha$.

\begin{proof}
The tangent of the curve $\beta$ is%
\begin{equation}
\beta^{\prime}(s)=((1-\lambda)\kappa+a)T+(c+\lambda\tau)B+(\kappa
(as+b)-\tau(cs+d))N.
\end{equation}

The tangent of the curve $(T)$ is%
\[
\frac{dT}{ds_{T}}=N.
\]
Here, $s_{T}$ is the arc parameter of the curve $(T).$%
\begin{equation}
\left\langle \beta^{\prime},N\right\rangle =\kappa(as+b)-\tau(cs+d)
\end{equation}
If the curves $\beta$ and $(T)$ are the involute-evolute pair then
\[
\left\langle \beta^{\prime},N\right\rangle =0.
\]
From (11), it can be easily obtained
\[
\frac{\kappa}{\tau}=\frac{cs+d}{as+b}%
\]
This means that the curve $\alpha$ is generalized Euler spiral.

On the other hand, if the curve $\alpha$ is generalized Euler spiral which has
the property%
\[
\frac{\kappa}{\tau}=\frac{cs+d}{as+b},
\]
then from (11)%
\[
\left\langle \beta^{\prime},N\right\rangle =0.
\]
This means that $\beta$ and $(T)$ are the involute-evolute pair.
\end{proof}
\end{theorem}

\textbf{Result.2 }From the hypothesis of the theorem above, the curve $\alpha$
is generalized Euler spiral which has the property%
\[
\frac{\kappa}{\tau}=\frac{cs+d}{as+b}%
\]
if and only if $\beta$ and $(B)$ are the involute-evolute pair. Here, $(B)$ is
the binormal of the curve $\alpha.$

\begin{proof}
The tangent of the curve $(B)$ is
\[
\frac{dB}{ds_{B}}=-N.
\]
The smiliar proof above in \textbf{Theorem 6 } can be given for $(B)$ again.

\begin{theorem}
Let the ruled surface $\Phi$ be%
\begin{align}
\Phi &  :I\times%
\mathbb{R}
\longrightarrow E^{3}\\
(s,v)  &  \rightarrow\Phi(s,v)=\alpha(s)+v[(as+b)T+(cs+d)B]\nonumber
\end{align}
The ruled surface $\Phi$ is developable if and only if the curve $\alpha$ is
generalized Euler spiral which has the property%
\[
\frac{\kappa}{\tau}=\frac{cs+d}{as+b}.
\]
Here, $\alpha$ is the base curve and $T,$ $B$ are the tangent and binormal of
the curve $\alpha,$ respectively.
\end{theorem}
\end{proof}

\begin{proof}
For the directrix of the surface
\[
X(s)=(as+b)T+(cs+d)B,
\]
and also for
\[
X^{\prime}(s)=aT+[(as+b)\kappa-(cs+d)\tau]N+cB,
\]
we can easily give that%
\[
\det(T,X,X^{\prime})=\left\vert
\begin{array}
[c]{ccc}%
1 & 0 & 0\\
as+b & 0 & cs+d\\
a & (as+b)\kappa-\tau(cs+d) & c
\end{array}
\right\vert
\]
In this case, the ruled surface is developable if and only if $\det
(T,X,X^{\prime})=0$ then%
\[
(cs+d)(as+b)\kappa-(cs+d)\tau=0
\]
then for $cs+d\neq0$%
\[
(as+b)\kappa-(cs+d)\tau=0.
\]
Thus, the curve $\alpha$ is generalized Euler spiral which has \ the property%
\[
\frac{\tau}{\kappa}=\frac{as+b}{cs+d}.
\]

\end{proof}

\begin{theorem}
Let $\alpha:I\rightarrow M$ be unit speed curve but not a general helix. If
the curve
\[
U(s)=\frac{1}{\kappa}T+\frac{1}{\tau}B
\]
is ta geodesic curve then the curve $\alpha$ is a logarithmic spiral in
$E^{3}.$
\end{theorem}

\begin{proof}
We have
\begin{align*}
U^{\prime}  &  =\left(  \frac{1}{\kappa}\right)  ^{\prime}T+\left(  \frac
{1}{\tau}\right)  ^{\prime}B\\
U^{\prime\prime}  &  =\left(  \frac{1}{\kappa}\right)  ^{\prime\prime
}T+\left(  \frac{1}{\tau}\right)  ^{\prime\prime}B+\left[  \left(  \frac
{1}{\kappa}\right)  ^{\prime}\kappa-\left(  \frac{1}{\tau}\right)  ^{\prime
}\tau\right]  N.
\end{align*}
Here,
\[
\left[  \left(  \frac{1}{\kappa}\right)  ^{\prime}\kappa-\left(  \frac{1}%
{\tau}\right)  ^{\prime}\tau\right]  \neq0
\]
and then
\begin{align*}
\left(  \frac{1}{\kappa}\right)  ^{\prime}\kappa-\left(  \frac{1}{\tau
}\right)  ^{\prime}\tau &  =-\frac{\kappa^{\prime}}{\kappa}+\frac{\tau
^{\prime}}{\tau}\\
&  =\frac{-\tau\kappa^{\prime}+\kappa\tau^{\prime}}{\kappa\tau}\\
&  =\left(  \frac{\kappa}{\tau}\right)  ^{\prime}\left(  \frac{\tau}{\kappa
}\right)  .
\end{align*}
If we take
\[
\left(  \frac{1}{\kappa}\right)  ^{\prime\prime}=0\text{ and }\left(  \frac
{1}{\tau}\right)  ^{\prime\prime}=0
\]
then the curve $\alpha$ is a logarithmic spiral. Here, it can be easily seen
that
\begin{align*}
\frac{1}{\kappa}  &  =as+b\\
\frac{1}{\tau}  &  =cs+d
\end{align*}
Thus, it is clear that the curve $\alpha$ is a logarithmic spiral.
\end{proof}

\section{CONCLUSIONS}

The starting point of this study is to extend the notion of Euler spirals to
$E^{3}$ by using some important properties. First, I introduce Euler spirals
in $E^{2}$ and in $E^{3}$ then define the generalized Euler spirals in $E^{3}%
$. Using these concepts, some necessary conditions for these curves to be
Euler spirals and generalized Euler spirals in $E^{3}$ are presented. Also,
some different characterizations of these curves are expressed by giving
theorems with their results. At this time, it is obtained that Euler spirals
in $E^{3}$ are Bertrand curves. Additionally, I show that all Euler spirals
are generalized Euler spirals in $E^{3}$ and also general helices are
generalized Euler spirals in $E^{3}.$ Moreover many different approaches about
generalized Euler spirals in Euchlidean 3-space are presented in this paper.

I hope that this study will gain different interpretation to the other studies
in this field.

\end{document}